\documentclass[11pt]{article}
\usepackage[margin=1in]{geometry}
\usepackage{amsmath,amssymb,amsthm,mathtools}
\usepackage[T1]{fontenc}
\usepackage{newtxtext,newtxmath}
\usepackage[colorlinks=true,linkcolor=blue,urlcolor=blue,citecolor=blue]{hyperref}
\usepackage{xurl}
\usepackage{enumitem}
\newtheorem{theorem}{Theorem}
\newtheorem{lemma}[theorem]{Lemma}
\newtheorem{corollary}[theorem]{Corollary}

\newcommand{\R}{\mathbb R}
\newcommand{\dd}{\,\mathrm d}
\newcommand{\norm}[1]{\left\|#1\right\|}
\newcommand{\G}{\Gamma}
\title{A special Hardy estimate in the axisymmetric swirl problem\\\large Research proof draft, revision 1.1}
\author{Daoguo Zhou\\
\normalsize School of Mathematics, Hangzhou Normal University, Hangzhou 311121, PR China\\
\normalsize \texttt{dgzhou@hznu.edu.cn}}
\date{21 October 2026}
\begin{document}
\maketitle
\begin{center}\small\textbf{Status: research proof draft, not a published theorem.}\\
The arguments below are submitted for mathematical verification.
\end{center}
\begin{abstract}
We examine the proof of Dongyi Wei's logarithmic circulation criterion. A trace estimate specific to $\nabla(u_r/r)$ replaces a generic squared-swirl form bound by one with one logarithm rather than two. Repeating the coupled $J$--$\Omega$ energy and comparison argument then gives a proposed threshold of order $(\log D)^{-5/4}$. Under the hypotheses stated below, the proof is intended to establish regularity when $|r u_\theta|\le C|\log r|^{-\alpha}$ uniformly near the axis with $\alpha>5/4$, and at $\alpha=5/4$ with an absolute sufficiently small coefficient. Every analytic step, including the cutoff core and exterior terms, is included for independent checking.
\end{abstract}

\section{Statement, normalization, and relation to Wei's criterion}
Consider the incompressible Navier--Stokes equations in $\R^3$ with viscosity one,
\begin{equation}\label{NS}
 \partial_t u+u\cdot\nabla u-\Delta u+\nabla p=0,\qquad \nabla\cdot u=0.
\end{equation}
Assume axisymmetry, and write
\[
 u=u_r e_r+u_\theta e_\theta+u_z e_z,\qquad
 \G=r u_\theta,\quad U=\frac{u_r}{r},\quad
 J=\frac{\omega_r}{r}=-\partial_z\frac{u_\theta}{r},\quad
 \Omega=\frac{\omega_\theta}{r}.
\]
For axisymmetric scalar functions we use $\dd x=r\dd r\dd z$; the omitted factor $2\pi$ only changes absolute constants. Vector and tensor norms are Cartesian norms.
Let $u\in C([0,T_*);H^2(\R^3))$ be the maximal strong solution with divergence-free axisymmetric initial data $u_0\in H^2$ and the additional assumption $\G_0\in L^\infty$. The estimates involving extra derivatives below are first performed at positive smooth times. Initial-time passage is explained at the end of the proof.

\begin{theorem}[Proposed quantitative criterion]\label{main}
There is an absolute constant $c_0>0$ with the following property. Define
\[
 M_1=(1+\norm{\G_0}_{\infty})\norm{u_0}_2,
 \qquad
 M_0=(\norm{J_0}_2+\norm{\Omega_0}_2)M_1^3,
 \qquad D=\max\{M_0^{1/4},r_0^{-1/2}M_1\}.
\]
For any $r_0>0$, if
\begin{equation}\label{threshold}
 \sup_{0<t<T_*}\sup_{0<r\le r_0,\ z\in\R}|\G(r,z,t)|
 \le c_0[1+\log(1+D)]^{-5/4},
\end{equation}
then $T_*=\infty$.
\end{theorem}

\begin{corollary}[Proposed logarithmic consequence]\label{cor}
Under the same initial-data and strong-solution hypotheses, suppose, for some $\delta_0\in(0,1/2)$,
\[
 |\G(r,z,t)|\le C_{\rm ax}|\log r|^{-\alpha}
 \quad (0<r\le\delta_0,\ z\in\R,\ 0<t<T_*).
\]
Then the proof below gives $T_*=\infty$ if $\alpha>5/4$ and $C_{\rm ax}$ is any fixed finite constant. It also gives this conclusion if $\alpha=5/4$ and $C_{\rm ax}$ is smaller than an absolute positive constant. Arbitrary coefficients at the endpoint are not asserted.
\end{corollary}

The data assumptions and the final continuation step are those used by Wei~\cite{Wei}. His Theorem~1.1 has exponent $3/2$ in the data-dependent logarithmic threshold, and his Corollary~1.1 states the near-axis condition $|\G|\le|\log r|^{-3/2}$. The change tested here is a specialized estimate for $u_\theta\nabla U$, not a change of constants in that theorem.

For viscosity $\nu>0$, set $\widetilde u(x,\tau)=\nu^{-1}u(x,\tau/\nu)$ and $\tau=\nu t$. Apply the viscosity-one statements to $\widetilde u$, so $\widetilde\G=\G/\nu$, $\widetilde J=J/\nu$, and $\widetilde\Omega=\Omega/\nu$. In particular the endpoint coefficient is an absolute constant times $\nu$; the data quantities in \eqref{threshold} are computed from the tilded initial data. The quantities $D$ and $\G$ are invariant under the usual viscosity-one spatial Navier--Stokes scaling.

\section{Preliminaries and a trace specific to the velocity quotient}
We use the energy inequality and circulation maximum principle
\begin{equation}\label{basic}
 \int_0^{T_*}\norm{\nabla u(t)}_2^2\dd t\le C\norm{u_0}_2^2,
 \qquad \norm{\G(t)}_\infty\le\norm{\G_0}_\infty=:G_0.
\end{equation}
The following elliptic estimates are Wei's Lemma~2.1:
\begin{equation}\label{elliptic}
 \norm{\nabla U}_2\le\norm{\Omega}_2,
 \qquad \norm{\nabla^2 U}_2\le\norm{\partial_z\Omega}_2.
\end{equation}
They follow from $(\Delta+2r^{-1}\partial_r)U=\partial_z\Omega$, integration by parts, and axis regularity $\partial_rU|_{r=0}=0$. The second derivative norm is the full Cartesian Hessian. In particular
\begin{equation}\label{radial}
 \int\frac{|\partial_r U|^2}{r^2}\dd x\le\norm{\nabla^2 U}_2^2.
\end{equation}
This is also immediate from the angular eigenvalue $\partial_rU/r$ of the horizontal Hessian.

\begin{lemma}[Vertical trace]\label{trace}
If $U\in H^2(\R^3)$, then
\begin{equation}\label{traceeq}
 \sup_{x_h\in\R^2}\int_{\R}|\partial_zU(x_h,z)|^2\dd z
 \le C\norm{\nabla^2U}_{L^2(\R^3)}^2.
\end{equation}
\end{lemma}
\begin{proof}
It suffices initially to take a Schwartz function. Let $\xi\in\R^2$ and $k\in\R$ be the horizontal and vertical Fourier variables. For $k\ne0$, Cauchy--Schwarz gives
\begin{align*}
 \left|k\int_{\R^2}e^{i x_h\cdot\xi}\widehat U(\xi,k)\dd\xi\right|^2
 &\le \left(\int_{\R^2}\frac{k^2}{(|\xi|^2+k^2)^2}\dd\xi\right)
 \left(\int_{\R^2}(|\xi|^2+k^2)^2|\widehat U(\xi,k)|^2\dd\xi\right).
\end{align*}
The first factor is an absolute constant. Integrating in $k$ and using Plancherel proves the estimate uniformly in $x_h$. Density proves the stated result, with the usual trace representative. The $H^2$ assumption fixes the affine ambiguity: a Hessian seminorm assumption alone would incorrectly admit $U=z$.
\end{proof}
For the actual $U=u_r/r$, physical $H^1$ velocity regularity controls $U\in L^2$ through the angular part of $\nabla u$. Together with \eqref{elliptic}, this supplies the $H^2$ hypothesis at positive smooth times.

Define the auxiliary quantities, distinguished from the velocity quotient $U$, by
\begin{equation}\label{aux}
 V(r,z,t)=\int_0^r|u_\theta(s,z,t)|\dd s,
 \qquad a_*(t)=\norm{V/r}_\infty.
\end{equation}
Wei's useful auxiliary estimate is
\begin{equation}\label{auxbound}
 a_*^2\le\norm{J}_2\norm{u_\theta/r}_2\le\norm{J}_2\norm{\nabla u}_2.
\end{equation}
For completeness, Cauchy--Schwarz in $r$, followed by integration of $-\partial_z(u_\theta^2/r)=2J u_\theta$ toward an end of the vertical line, gives
\[
 V(r,z)^2\le r^2\int |J|\,|u_\theta/r|\,s\dd s\dd z
 \le r^2\norm{J}_2\norm{u_\theta/r}_2.
\]
Approximation justifies this computation for the stated class.

\section{An arbitrary-parameter linear Hardy estimate}
Fix a smooth time, and suppose $|\G|\le\epsilon$ on $r\le r_0$, where $0<\epsilon\le1$. Let $K\ge2$ be arbitrary, and put
\[
 \rho=\epsilon/a_*,\qquad r_1=\min\{\rho K,r_0\},\qquad L=\log K.
\]
When $a_*=0$, swirl is identically zero and all the estimates below are trivial.

\begin{lemma}[Linear form]\label{linear}
For an axisymmetric scalar or vector function $f\in H^1$,
\begin{equation}\label{linearform}
 Q(f):=\int\frac{|u_\theta|}{r}|f|^2\dd x
 \le h\norm{\partial_rf}_2^2+
 \frac{C(G_0+h)}{r_1^2}\int_{r\ge r_1/2}|f|^2\dd x,
 \qquad h=C\epsilon(1+L)^2.
\end{equation}
The constant can be chosen absolute, independently of $K$.
\end{lemma}
\begin{proof}
For $f$ supported in $r<R\le r_1$, radial Cauchy--Schwarz and Fubini give
\begin{align*}
 \int_0^R |u_\theta(r,z)|\,|f(r,z)|^2\dd r
 &\le \left(\int_0^R s|\partial_rf(s,z)|^2\dd s\right)
 \left(\int_0^R\frac{V(s,z)}s\dd s\right).
\end{align*}
We have $V(s,z)\le a_*s$, and, for $\rho\le s\le r_1$,
\[
 V(s,z)\le\epsilon[1+\log(s/\rho)].
\]
Consequently
\[
 \int_0^{r_1}\frac{V(s,z)}s\dd s
 \le\epsilon(1+L+L^2/2).
\]
Integrate in $z$. For a general $f$, apply the supported estimate to $\chi(r/r_1)f$, where $\chi=1$ on $[0,1/2]$ and $\chi=0$ on $[1,\infty)$. The derivative of the cutoff contributes $C h r_1^{-2}\int_{r\ge r_1/2}|f|^2$. Outside the cutoff, use $|u_\theta|/r=|\G|/r^2\le4G_0/r_1^2$. Enlarging the absolute constant proves \eqref{linearform}.
\end{proof}
This proves the needed estimate for arbitrary $K$; it does not invoke the fixed choice $K(\epsilon)$ in the statement of Wei's Lemma~2.3.

\section{The specialized squared-swirl estimate}
\begin{lemma}[Special form bound]\label{special}
In the preceding notation, for an axisymmetric $U\in H^2(\R^3)$,
\begin{equation}\label{specialform}
 \int |u_\theta|^2|\nabla U|^2\dd x
 \le k\norm{\nabla^2U}_2^2+
 \frac{C(G_0^2+\epsilon^2)}{r_1^2}
 \int_{r\ge r_1/2}|\nabla U|^2\dd x,
 \qquad k=C\epsilon^2(1+L).
\end{equation}
\end{lemma}
\begin{proof}
First assume $r_1\ge2\rho$. Take a radial cutoff supported in $r<2\rho$ and equal to one in $r<\rho$. The proof of Lemma~\ref{linear}, on this short radius, gives
\[
 \int\frac{|u_\theta|}{r}|f|^2\dd x
 \le C\epsilon\int|\partial_rf|^2\dd x
 \quad\text{if }\operatorname{supp}f\subset\{r<2\rho\}.
\]
Since $|u_\theta|^2\le\epsilon |u_\theta|/r$ on this region, applying it to the cutoff times $\nabla U$ yields
\begin{equation}\label{core}
 \int_{r<\rho}|u_\theta|^2|\nabla U|^2\dd x
 \le C\epsilon^2\left(\norm{\nabla^2U}_2^2+
 \rho^{-2}\int_{\rho<r<2\rho}|\nabla U|^2\dd x\right).
\end{equation}
The radial component of the shell term is bounded by \eqref{radial}. For the vertical component, Lemma~\ref{trace} gives
\[
 \rho^{-2}\int_\rho^{2\rho}r\dd r\int|\partial_zU(r,z)|^2\dd z
 \le C\norm{\nabla^2U}_2^2.
\]
Thus \eqref{core} is at most $C\epsilon^2\norm{\nabla^2U}_2^2$.

On $\rho\le r\le r_1$, use $|u_\theta|\le\epsilon/r$. The radial component is again controlled by \eqref{radial}. The vertical component satisfies
\[
 \int_\rho^{r_1}\frac{\epsilon^2}{r}\dd r
 \int|\partial_zU(r,z)|^2\dd z
 \le C\epsilon^2\log(r_1/\rho)\norm{\nabla^2U}_2^2
 \le C\epsilon^2L\norm{\nabla^2U}_2^2.
\]
This treats the middle annulus.

If instead $r_1<2\rho$, use a cutoff supported in $r<r_1$ and equal to one for $r<r_1/2$. Since $a_*r_1<2\epsilon$, the supported linear estimate has coefficient at most $2\epsilon$. Its squared-swirl consequence gives $C\epsilon^2\norm{\nabla^2U}_2^2$ and the allowed shell remainder $C\epsilon^2r_1^{-2}\int_{r\ge r_1/2}|\nabla U|^2$.

Finally, in either case, the exterior region is bounded using $|u_\theta|^2\le4G_0^2/r_1^2$ for $r\ge r_1/2$. Combining the core, shell, annulus, and exterior bounds proves \eqref{specialform}.
\end{proof}
The improvement over the generic bound is exactly one power of $1+\log K$. Both the angular Hessian term and the vertical trace are essential.

\section{Coupled energy and absorption without circular constants}
The equations for the two vorticity quotients are
\begin{align}
 \partial_t\Omega+u\cdot\nabla\Omega-(\Delta+2r^{-1}\partial_r)\Omega
 &=-2(u_\theta/r)J,\label{Omegaeq}\\
 \partial_tJ+u\cdot\nabla J-(\Delta+2r^{-1}\partial_r)J
 &=(\omega_r\partial_r+\omega_z\partial_z)U.\label{Jeq}
\end{align}
For smooth axis-compatible fields the diffusion contributes the favorable axis terms. Set
\[
 Y=\norm{J}_2^2,\quad Z=\norm{\Omega}_2^2,\quad
 D_J=\norm{\nabla J}_2^2,\quad D_\Omega=\norm{\nabla\Omega}_2^2.
\]
Integration by parts gives
\begin{align}
 \tfrac12Y'+D_J+\int_\R|J(0,z)|^2\dd z
 &=\int J(\omega_r\partial_r+\omega_z\partial_z)U\dd x,\label{Jenergy}\\
 \tfrac12Z'+D_\Omega+\int_\R|\Omega(0,z)|^2\dd z
 &=-2\int(u_\theta/r)J\Omega\dd x.\label{Oenergy}
\end{align}
Using $\omega_r e_r+\omega_z e_z=\nabla\times(u_\theta e_\theta)$, the first right-hand side is
\[
 \int u_\theta e_\theta\cdot(\nabla J\times\nabla U)\dd x
 \le\tfrac12D_J+\tfrac12\norm{u_\theta\nabla U}_2^2.
\]
After multiplying by two and discarding only nonnegative axis terms,
\begin{equation}\label{JE}
 Y'+D_J\le kD_\Omega+
 \frac{C(G_0^2+\epsilon^2)}{r_1^2}
 \int_{r\ge r_1/2}|\nabla U|^2\dd x.
\end{equation}
Here Lemma~\ref{special} and \eqref{elliptic} have been used.

For any $\sigma>0$, $2|J\Omega|\le\sigma J^2+\sigma^{-1}\Omega^2$. Thus \eqref{Oenergy} and Lemma~\ref{linear} imply
\begin{align}
 Z'+2D_\Omega
 &\le2\sigma hD_J+2\sigma^{-1}hD_\Omega\notag\\
 &\quad+\frac{C(G_0+h)}{r_1^2}
 \int_{r\ge r_1/2}(\sigma J^2+\sigma^{-1}\Omega^2)\dd x.\label{OE}
\end{align}
Choose $\sigma=4h$ and $\lambda=2k$. Adding \eqref{JE} to $\lambda$ times \eqref{OE} leaves positive coefficients of both dissipations provided
\begin{equation}\label{absorb}
 16kh^2\le\tfrac12.
\end{equation}
Indeed, the $D_J$ coefficient is at least $1-16kh^2$, and the $D_\Omega$ coefficient is $3k-k=2k>0$.

The constants in $h=C_h\epsilon(1+L)^2$ and $k=C_k\epsilon^2(1+L)$ are fixed by the two lemmas before choosing parameters. Choose an absolute $b>0$ sufficiently small and then an absolute $\epsilon_*>0$ sufficiently small, and set
\begin{equation}\label{choose}
 L=b\epsilon^{-4/5},\quad K=e^L,\quad0<\epsilon\le\epsilon_*.
\end{equation}
Then $K\ge2$, \eqref{absorb} holds, $\lambda\le1$, and, with fixed absolute comparison constants,
\begin{equation}\label{powers}
 h\asymp\epsilon^{-3/5},\qquad k\asymp\lambda\asymp\epsilon^{6/5}.
\end{equation}
There is no dependence of $b$ or $\epsilon_*$ on the solution or its initial data.

\section{Exterior control and the comparison inequality}
Define $E=Y+\lambda Z$ and $F(t)=\norm{\nabla u(t)}_2$. The remaining exterior coefficients multiply, respectively, $|\nabla U|^2,J^2,\Omega^2$, and are bounded by
\[
 C(G_0^2+\epsilon^2),\qquad
 C\lambda\sigma(G_0+h),\qquad
 C\lambda\sigma^{-1}(G_0+h).
\]
The global bounds $\norm{\nabla U}_2^2\le Z$, $Y\le E$, and $Z\le E/\lambda$ yield an exterior integral bound $CM E$, where we may take
\begin{equation}\label{M}
 M=C\epsilon^{-6/5}(1+G_0)^2.
\end{equation}
Independently, on $r\ge r_1/2$, the definitions and the cylindrical expression for $|\nabla u|^2$ imply
\begin{equation}\label{exterior}
 \int_{r\ge r_1/2}(|\nabla U|^2+J^2+\Omega^2)\dd x
 \le Cr_1^{-2}F^2.
\end{equation}
For example $\nabla U=\nabla u_r/r-u_r e_r/r^2$, and $|u_r|/r$ is part of the Cartesian velocity-gradient norm. The same $M$ bounds all three exterior coefficients without energy weights. Consequently both estimates hold simultaneously, giving
\begin{equation}\label{min}
 E'\le\frac{CM}{r_1^2}\min\{E,F^2/r_1^2\}.
\end{equation}
All parameters here are fixed in time except $a_*$ and hence $r_1$; no differentiated time-dependent energy weight has been introduced.

Put $d=\epsilon K$. If $r_1=d/a_*$, \eqref{auxbound} gives $r_1^{-2}\le E^{1/2}F/d^2$. Therefore
\begin{align*}
 E'&\le\frac{CM EF}{d^2}
 \min\{E^{1/2},F^3/d^2\}\\
 &\le\frac{CM E^{4/3}F^2}{d^{8/3}},
\end{align*}
where $\min\{A,B\}\le A^{2/3}B^{1/3}$ was used. If $r_1=r_0$, \eqref{min} instead gives $E'\le CM F^2/r_0^4$. Hence
\begin{equation}\label{ode}
 E'\le CM F^2\max\{E^{4/3}/d^{8/3},r_0^{-4}\}.
\end{equation}
Let
\[
 \Phi(y)=\int_y^\infty\frac{\dd s}{\max\{s^{4/3}/d^{8/3},r_0^{-4}\}}.
\]
Its derivative is negative, so \eqref{ode} implies the correctly signed inequality
\[
 \frac{\dd}{\dd t}\Phi(E(t))\ge-CM F(t)^2.
\]
Using \eqref{basic}, it suffices to have
\begin{equation}\label{comparison}
 \Phi(E(0))>CM\norm{u_0}_2^2
\end{equation}
to keep $E$ bounded. Indeed $\Phi(E(t))$ then stays bounded away from zero and $\Phi(y)\to0$ as $y\to\infty$.

An explicit lower bound is
\[
 \Phi(E(0))\ge3d^{8/3}
 \max\{E(0),d^2/r_0^3\}^{-1/3}.
\]
Thus a sufficiently large absolute constant in the following two conditions guarantees \eqref{comparison}:
\[
 d^8\ge C E(0)(M\norm{u_0}_2^2)^3,
 \qquad d^2r_0\ge CM\norm{u_0}_2^2.
\]
Since $\lambda\le1$, $E(0)^{1/2}\le\norm{J_0}_2+\norm{\Omega_0}_2$. By \eqref{M} and the definitions of $M_0,M_1$, it is enough that
\begin{equation}\label{Kcriterion}
 K\ge C\epsilon^{-8/5}\max\{M_0^{1/4},r_0^{-1/2}M_1\}=C\epsilon^{-8/5}D.
\end{equation}
The first of the preceding two conditions actually needs only the smaller polynomial $\epsilon^{-29/20}$.

\section{Completion of the threshold and the continuation step}
Choose
\[
 \epsilon=c_0[1+\log(1+D)]^{-5/4}.
\]
Then
\[
 e^{b\epsilon^{-4/5}}=e^{b c_0^{-4/5}}(1+D)^{b c_0^{-4/5}},
 \qquad
 \epsilon^{-8/5}=c_0^{-8/5}[1+\log(1+D)]^2.
\]
Choosing the absolute $c_0$ sufficiently small makes $b c_0^{-4/5}\ge3$, dominates the fixed constant in \eqref{Kcriterion}, and ensures $\epsilon\le\epsilon_*$. The elementary inequality
$D[1+\log(1+D)]^2\le C(1+D)^3$ then verifies \eqref{Kcriterion} for all $D\ge0$.
We have proved $\sup_{t<T_*}\norm{\Omega(t)}_2<\infty$.

This last statement is not silently identified with an $H^2$ bound for the full velocity. The direct continuation criterion of Chen--Fang--Zhang~\cite[Lemma~3.1, printed p.~10]{CFZ} states that a maximal axisymmetric strong solution in $C([0,T_*);H^2)\cap L^2_{\mathrm{loc}}([0,T_*);H^3)$ with $\omega_\theta/r\in L^\infty(0,T_*;L^2)$ extends past any finite $T_*$. Their symbol $\Gamma$ in that lemma denotes our $\Omega$, not our circulation $r u_\theta$. Positive-time strong-solution regularity supplies the stated local $H^3$ property. The bound just proved therefore excludes a finite maximal existence time, and $T_*=\infty$. This is the continuation route cited in Wei's introduction, now identified explicitly; no additional weighted initial-data hypothesis is used.

To justify starting at zero, one may perform every extra-derivative computation on $[\tau,t]$ with $\tau>0$ and then let $\tau\downarrow0$. In the axisymmetric $H^2$ class, $J,\Omega\in C([0,T_*);L^2)$, as in Wei's framework. The energy weights are fixed, the initial norms converge, and the dissipative inequalities pass by lower semicontinuity. This completes the proposed proof of Theorem~\ref{main}.

For Corollary~\ref{cor}, as $r_0\downarrow0$ with fixed nonzero data,
\[
 1+\log(1+D)=\tfrac12\log(1/r_0)+O_{u_0}(1).
\]
If $\alpha>5/4$, the assumed near-axis bound is eventually smaller than \eqref{threshold}, whatever the fixed finite $C_{\rm ax}$. At $\alpha=5/4$, the same comparison holds for a sufficiently small absolute coefficient. Zero data are trivial.

\section{Scope and points for further checking}
\begin{itemize}[leftmargin=*]
 \item This is a proof draft derived from Wei's strong-solution framework.
 \item The near-axis bound is uniform in $z$ and throughout the maximal existence interval. No conclusion is asserted from a bound at a single time or from an averaged-in-time substitute.
 \item The specialized trace requires the normalization supplied by $U\in H^2$; it is not a trace estimate for an arbitrary $H^1$ vector field.
 \item The global circulation maximum principle, finite energy, and the exterior remainders are retained. The argument does not turn a near-axis smallness hypothesis into arbitrary global-small-circulation regularity.
 \item The endpoint $5/4$ requires an absolute small coefficient. Improving below $5/4$ is not proved here.
\end{itemize}

\noindent\emph{Revision note (1.1).} The theorem and all estimates are unchanged. This revision identifies the exact direct continuation lemma used in the last step; it does not claim formal certification or external peer review.

\end{document}